\documentclass[sigconf]{acmart}

\newcommand{\RR}{{\mathbb R}}

\newcommand{\xx}{{\mathbf x}}

\newcommand{\ff}{{\mathbf f}}

\newcommand{\pp}{{\mathbf p}}

\usepackage{algorithm}
\usepackage{algpseudocode} % provides algpseudocode (algorithmicx)
\usepackage{graphicx}
\usepackage{subfigure}
\usepackage{booktabs}
\usepackage{multirow}
\usepackage{hyperref}
\usepackage{color}
\usepackage{algorithm}
\usepackage{algorithmicx}
\usepackage{algpseudocode}
\usepackage{amsmath}
\usepackage{listings}
\usepackage{xcolor}
\AtBeginDocument{%
  }

\setcopyright{acmlicensed}
\copyrightyear{2026}
\acmYear{2026}
\acmDOI{XXXXXXX.XXXXXXX}
\renewcommand\footnotetextcopyrightpermission[1]{}
\acmISBN{978-1-4503-XXXX-X/2018/06}

\begin{document}

% \title{Formal Safety Verification for Nonlinear Systems with\\ Generative Barrier Certificate}
\title[Formal Safety Verification with
Generative Barrier Certificate]{Formal Safety Verification for Nonlinear Systems with Generative Barrier Certificate}

\author{Mengxin Ren\textsuperscript{1}, Hanrui Zhao\textsuperscript{2}}
\affiliation{%
 \institution{\textit{$^{1}$Shanghai Key Lab of Trustworthy Computing, East China Normal University, Shanghai, China}\\
\textit{$^{2}$College of Computer Science and Technology, National University of Defense Technology, Changsha, China}
}
\country{}
}

\renewcommand{\shortauthors}{Ren et al.}

\begin{abstract}
Safety verification is a fundamental problem in control theory. Barrier certificates (BCs) provide a powerful formal mechanism, yet deriving BCs is computationally intensive.
This paper introduces a generative framework that leverages large language models (LLMs) to synthesize BCs through reasoning. 
Based on the classical Sum-of-Squares (SOS) approach, we train a domain-specific LLM capable of generating high-quality BC candidates for nonlinear systems. Then, the LLM-generated BCs transform the intractable Bilinear Matrix Inequality (BMI) solving problems into convex Linear Matrix Inequality (LMI) feasibility test, significantly improving efficiency while preserving correctness. 
Experimental results show that our generative method achieves several orders of magnitude speedup over traditional numerical BC approaches and, perhaps surprisingly, surpasses the state-of-the-art dedicated neural BC model. These findings mark a substantive step toward integrating generative AI with formal safety verification for dynamical systems.
\end{abstract}



\keywords{Formal Verification, Barrier Certificate, Large Language Model,
Generative Artificial Intelligence}


\maketitle

\section{Introduction}
Safety assurance in safety-critical domains is essential for various practical applications including autonomous vehicles and robotics, as erroneous decisions may result in catastrophic failures~\cite{Mirzarazi,perez2024artificial,guiochet2017safety}. Conventional verification paradigms typically rely on simulation-based testing~\cite{li2023simulation}, which is useful but cannot exhaustively explore all system behaviors. Formal verification~\cite{hasan2015formal,berard2013systems,bresolin2015formal} addresses this limitation by providing mathematically rigorous guarantees that system trajectories satisfy safety constraints under all admissible conditions. The fundamental goal is to certify that, from any state within a prescribed initial set, the system will never reach designated unsafe regions.

Barrier certificate (BC)~\cite{prajna2004safety} is a promising formal method for safety verification. A feasible BC separates all system trajectories from the unsafe region, yielding an over-approximation of the reachable set and certifying safety over an infinite horizon. BCs synthesis requires global sign-definiteness which aligns closely with quantifier-elimination paradigms~\cite{sturm2011verification}. Among existing approaches, sum-of-squares (SOS) relaxation~\cite{prajna2007framework} has attracted significant attention due to its capability to mitigate computational complexity~\cite{prajna2007framework,kong2013exponential}. The core idea is to construct SOS programs that transform quantified nonlinear constraints into nonconvex bilinear matrix inequalities (BMIs), which are subsequently handled by semidefinite programming solvers such as PENBMI~\cite{kocvara2005penbmi}. 
% Prajna et al. pioneered a SOS-based BCs synthesis framework for semialgebraic hybrid systems~\cite{prajna2007framework}, and Kong et al. further advanced this direction by introducing exponential-condition-based SDP formulations~\cite{kong2013exponential}.  
However, BMI feasibility is known to be NP-hard, which fundamentally limits scalability. 
To improve tractability, convex surrogate methods fix multipliers to convert BMIs into linear matrix inequalities (LMIs)~\cite{boyd1994linear}, enabling efficient convex optimization as implemented in SOSTOOLS~\cite{papachristodoulou2013sostools}, albeit at the cost of conservativeness that may eliminate feasible solutions to the original BMI problem. Thus, designing BCs parameterizations that balance expressiveness and computational efficiency remains an challenging problem.

With the rapid progress of machine learning, learning-based BCs synthesis~\cite{zhao2020synthesizing} has emerged as a powerful complement to traditional SOS methods. Existing studies typically employ feedforward neural networks to represent barrier functions and utilize modern training techniques to significantly improve scalability and applicability~\cite{zhao2023formal,abate2021fossil}. However, such approaches often require task-specific network architectures trained from scratch and rely on predefined functional templates (e.g., fixed-degree polynomials), which inherently limit model expressiveness and constrain the searchable solution space, thereby hindering the discovery of more optimal or non-prespecified  forms. Moreover, these learning-based methods still depend on explicit structural parameterization, making it difficult to achieve truly open-ended function discovery.

Against this backdrop, the advent of large language models (LLMs)~\cite{zhao2023survey,naveed2025comprehensive,chang2024survey,liu2024llms} introduces a paradigm shift. Unlike task-specialized neural architectures, LLMs serve as general-purpose knowledge and reasoning engines pretrained on massive corpora, endowed with nontrivial mathematical, logical, and symbolic reasoning capabilities. This enables the possibility of breaking away from fixed functional templates and rigid training architectures, and moving toward generative, adaptive, and more expressive BC synthesis.
Building on these insights, in this paper, we introduce a novel paradigm that leverages end-to-end fine-tuning of LLMs to generate BCs. Specifically, we reformulate the BC synthesis problem as a sequence-to-sequence generation task, where the input is a structured representation of the dynamical system and the output is the explicit symbolic form of a valid BC. By treating synthesis as a generative modeling problem, this approach unifies the traditionally complex optimization and constraint-solving processes into a single, learnable framework, enabling the discovery of BCs beyond fixed functional templates and alleviating the scalability limitations of conventional methods.

The main contributions of this paper can be summarized as follows:
\begin{itemize}
    % \item We propose and implement a novel framework that solves the barrier certificate synthesis problem end-to-end by fine-tuning a large language model.
    % % \item Our method represents a paradigm shift from parameter optimization to symbolic creation, which breaks free from fixed functional constraints to endow the model with a more powerful problem-solving capability for discovering diverse solutions.
    % \item Our method shifts from parameter optimization to symbolic generation, overcoming fixed functional constraints and enabling the discovery of diverse solutions.
    \item We develop an end-to-end framework that fine-tunes a LLM to synthesize barrier certificates through symbolic generation rather than parameter optimization, thereby overcoming template constraints and enabling the discovery of diverse and expressive ones.
    % \item We demonstrate how a pre-trained general knowledge engine can be adapted to solve specific scientific computing problems, rather than building specialized tools from scratch, which significantly enhances learning efficiency and the model's generalization potential.
    % \item A pre-trained general knowledge engine is adapted to specific scientific computing tasks enhances learning efficiency and cross-system generalization.

    \item We transform a pre-trained general-purpose language model into a domain-specialized scientific reasoning engine using an inverse-design corpus with diffeomorphic data augmentation, substantially improving learning efficiency and cross-system generalization.

    % \item Experimental results on a series of benchmarks show that our fine-tuned model can efficiently and accurately generate valid BCs for complex non-linear systems, surpassing existing methods in terms of success rate and applicability, and showcasing its potential as a unified and powerful framework for scientific discovery.
    % \item Experimental results demonstrate that our fine-tuned LLM for generative BC synthesis outperforming existing numerical and learning-based methods in both success rate and efficiency.
    \item Experimental results demonstrate that our fine-tuned LLM for generative BC synthesis outperforming existing numerical and neural approaches in both success rate and efficiency.
\end{itemize}

%%%%改下面的！！！！！%%%%%
% The organization of this paper is as follows. Preliminaries are described in Section 2. We discuss each part of the proposed framework in Section 3. In Section 4, we evaluate our algorithms through experiments. Section 5 reviews related work and we conclude the paper in Section 6.

% The organization of this paper is as follows. Preliminaries are described in Section 2. We present our novel framework in Section 3, detailing the end-to-end process from data synthesis to model fine-tuning and verification. In Section 4, we present our experimental results, including a detailed case study and a quantitative performance comparison. Finally, we conclude the paper and discuss future work in Section 5.
\section{Preliminaries}\label{sec:Pre}
% \textbf{Notations.}
%  Let $\RR$ and $\NN$ be the field of real numbers
%  and natural numbers, respectively.
%  $\RR[\xx]$ denotes the ring of polynomials with coefficients in $\RR$
%  over variables $\xx=[x_1,x_2,\ldots,x_n]^T$,
%  and $\RR[\xx]^n$ denotes the
%  $n$-dimensional polynomial ring vector.
% Let $R[\xx]_{d} \subset \RR[\xx]$ be the vector space of polynomials
%  of degree
%  at most $d$. %Let $\NN_{d}^{n}:=\{\alpha \in \NN^{n}: \sum_{i} \alpha_i \leq d\}$.
% Denote by $\Sigma[\xx]\subset \RR[\xx]$
%  (resp. $\Sigma[\xx]_d \subset \RR[\xx]_{2d}$)
%  the space of sums of squares (SOS) polynomials.
\textbf{Notations.}
Let $\mathbb{R}$ and $\mathbb{N}$ denote the sets of real and natural numbers, respectively.
The symbol $\mathbb{R}[\mathbf{x}]$ represents the polynomial ring in variables $\mathbf{x} = [x_1, x_2, \ldots, x_n]^{T}$ with coefficients in $\mathbb{R}$, while $\mathbb{R}[\mathbf{x}]^n$ denotes the $n$-dimensional vector space consisting of such polynomial entries.
$R[\xx]_{d} \subset \RR[\xx]$ denotes the subspace of all polynomials whose total degree does not exceed $d$.
The set of sum-of-squares (SOS) polynomials is denoted by $\Sigma[\xx]\subset \RR[\xx]$, and $\Sigma[\xx]_d \subset \RR[\xx]_{2d}$ represents the subset of SOS polynomials with degrees up to $2d$.
\subsection{Problem Statements}

%This section introduces the definition of a continuous system and the safety verification conditions, and the description of relevant symbols is given below first.

%Let $\RR$ and $\NN$ be the field of a real number and natural number,
%respectively. The symbol $\RR[\xx]:=\RR[x_1,\ldots,x_n]$ denotes the polynomial ring with coefficients in $\RR$ over variable $\xx=[x_1,x_2,\ldots,x_n]^{T}$, and $\RR[\xx]^m$ denotes
%the $m$-dimensional polynomial ring vector.
%For a polynomial $p(\xx) \in \RR[\xx]$,
%we use $\deg(p)$ to denote its total degree. %%%
%For $d\in \NN$, $\RR[\xx]_{d}$ denotes the set consisting of all real polynomials in $\xx$ with degree at most $d$. And
%for a polynomial ring vector $p=[p_1,\ldots,p_m]^{T} \in \RR[\xx]^m$,
%we use $\deg(p)$ to denote the highest degree among $p_{i}$, i.e., $\deg(p)=\max\{\deg(p_1),\ldots,\deg(p_m)\}$.
%Denote by $\Sigma[\xx]\subset \RR[\xx]$
%$\Sigma[\xx]_d$ denotes the $d$-th truncation of $\Sigma[\xx]$, i.e.,
%$ \Sigma[\xx]_d:= \Sigma[\xx] \cap \RR[\xx]_{d}$.

Considering a dynamical system by a finite number of first-order ordinary differential equations:
%For our purpose of safety controller synthesis, consider using differential equations to define the continuous dynamical system of the form
\begin{eqnarray}\label{sys}
\dot{\xx}=\ff(\xx),
\end{eqnarray}
% where $\dot{\xx}$ denotes the derivative of $\xx=[x_1,\ldots,x_n]^T\in\RR^n$ with respect to
% the time variable $t$, and  $\ff=[f_1,\ldots,f_n]^T$ is the vector field %%% 
% on the state space $\Psi\subset\RR[\xx]^n$. Assume that $\ff$ satisfies the local Lipschitz
% condition, which ensures that given the initial state point $\xx_0\in \Psi$,
% there exists a time $T>0$ and a unique function $\tau: [0,T) \mapsto \RR^{n}$ such that $\tau(0)=\xx_0$. And $\xx(t)$ is called a trajectory of (\ref{sys}) from $\xx_0$.
where $\dot{\xx}$ denotes the derivative of the state vector $\xx=[x_1,\ldots,x_n]^T\in\RR^n$  with respect to
the time variable $t$, and $\ff=[f_1,\ldots,f_n]^T$ denotes the vector field defined on the state space $\Psi\subset\RR[\xx]^n$.
We assume that $\ff$ is locally Lipschitz continuous, ensuring the existence of a time horizon $T > 0$ and a unique trajectory $\boldsymbol{\tau}: [0, T) \to \mathbb{R}^n$ satisfying $\boldsymbol{\tau}(0) = \mathbf{x}_0$ for any initial condition $\mathbf{x}_0 \in \Psi$. The trajectory $\mathbf{x}(t)$ of system (\ref{sys}) starting from $\mathbf{x}_0$ is thus well-defined.

For formal exposition, we define the dynamical system over $\xx$ as a tuple $\mathcal{F}:  \langle \ff, \Theta, \Psi \rangle$, wherein $\Theta$ is a set of initial states, $\ff$ is a vector field over the domain $\Psi \subset \RR^{n}$. $\Theta$ and $\Psi$ are expressed as semi-algebraic sets with polynomial equalities and inequalities over the system variables. For clarity, we denote $\Theta, \Psi$ as:
\begin{equation*}
\begin{array}{l}
\Psi:=\{\xx \in \RR^{n}\,|\,\psi_{1}(\xx)\geq 0,\ldots,\psi_{\kappa}(\xx)\geq 0\},\\
\Theta:=\{\xx \in \RR^{n}\,|\,\theta_1(\xx)\geq 0,\ldots, \theta_{\iota}(\xx)\geq 0\},
\end{array}
\end{equation*}
where $\psi_{j} \in \RR[\xx], 1 \leq j \leq \kappa$ and  $\theta_{i} \in \RR[\xx], 1\leq i \leq \iota$. Furthermore, these semi-algebraic sets are assumed to be compact. Suppose that the system has an unsafe region $\Xi$ which is a compact semi-algebraic set:
%Let $\Xi$ be a pre-specified unsafe region, defined by a semi-algebraic set:
\begin{equation*} \label{eq:X_u}
 \Xi:=\{\xx \in \RR^{n} \,| \, \xi_{1}(\xx)\geq 0,\ldots, \xi_{r}(\xx)\geq 0\},
\end{equation*}
where $\xi_{i}\in \RR[\xx], 1\leq i \leq r$. The safety of dynamical system $\mathcal{F}$ can be defined with respect to the unsafe region $\Xi$.

\begin{definition}[Safety]
For a dynamical system $\mathcal{F}:  \langle \ff, \Theta, \Psi \rangle$ with a predefined unsafe region represented by the state set $\Xi$, we say that the system $\mathcal{F}$ is {\it safe} if no trajectory originating from the initial set $\Theta$ enters the unsafe region $\Xi$. Formally, for any initial state $\mathbf{x}_0 \in \Theta$, there exists no time $t_1 > 0$ such that
$$ \forall t \in [0, t_1]. \mathbf x(t, \mathbf x_0) \in \Psi\ \,\,\, \mathrm{and} \,\, \mathbf x(t_1, \mathbf x_0)\in \Xi. $$
\end{definition}

Barrier certificate (BC) provides a sufficient condition for the safety property defined above, which is a real-valued function that separate the reachable set from the unsafe region, assigning non-negative values to all reachable states and negative values within the unsafe set. Here, we adopt a relaxed verification condition by introducing auxiliary polynomials for BCs synthesis adapted from~\cite{prajna2004safety}, preserves convexity while expanding the feasible space, facilitating more effective reasoning-based search.
\begin{theorem}\label{th:BC} %[Barrier Certificate]\label{prajna2007safety}
Let $\mathcal{F}: \langle \ff, \Theta, \Psi \rangle $ be a continuous system and an unsafe state set $\Xi$. If there exists a real-valued function $B: \Psi\rightarrow \RR$ and a polynomial $\lambda(\xx)$, 
%which satisfies the following conditions: 
%a barrier certificate of $\mathcal{C}$ for safety with respect to the unsafe region $\Xi$ is a real-valued function $B: \Psi\rightarrow\RR$,
such that the following conditions hold:
\begin{description}
\item[(i)] $B(\xx)\leq 0\quad\forall  \xx\in\Theta$,
\item[(ii)] $\mathcal{L}_f B(\xx) -\lambda(\xx) B(\xx)<0 \; \forall \xx \in \Psi$, where $\mathcal{L}_f B(\xx)$ denotes the {Lie-derivative} of $B(\xx)$ along the vector field $\ff(\xx)$, i.e., $\mathcal{L}_f B(\xx)=\sum_{i=1}^n\frac{\partial B}{\partial x_i} \cdot f_{i}(\xx)$,
\item[(iii)] $B(\xx)> 0\quad\forall \xx\in \Xi$,
%\item[(iii)] $B(\xx)=0\Rightarrow\mathcal{L}_f B(\xx)> 0\quad\forall \xx\in \Psi$,
\end{description}
%where $\mathcal{L}_f B(\xx)$ denotes the {Lie-derivative} of %$B(\xx)$ along the vector field $\ff(\xx)$, i.e., $\mathcal{L}_f %B(\xx)=\sum_{i=1}^n\frac{\partial B}{\partial x_i} \cdot f_{i}%%(\xx)$. In this paper, the third condition is relaxed to %$\mathcal{L}_f B(\xx)-\lambda B(\xx)<0\quad\forall \xx\in \Psi$, %which is convex.
then $B(\xx)$ is a BC of the system $\mathcal{F}$, and the safety of $\mathcal{F}$ is guaranteed.
\end{theorem}
%\begin{remark}
% Observing Condition~(ii), it is easy to see that $\mathcal{L}_f B(\xx)<0$ if $B(\xx)=0$, which, combining  with Condition (i) and (ii), indicates that $B(\xx)$ is a barrier certificate of the system $\mathcal{F}$ with respect to the given unsafe state set.
% \begin{proof}
{\em Proof.}
From Condition~(ii), it follows that $\mathcal{L}_f B(\mathbf{x}) < 0$ whenever $B(\mathbf{x}) = 0$. Together with Condition~(i) and (iii), this implies that $B(\mathbf{x})$ satisfies the requirements of a barrier certificate for the system $\mathcal{F}$ relative to the specified unsafe set. $ \hfill \square$
\subsection{Sum-of-Squares Relaxation}

We briefly review the standard machinery of Sum-of-Squares (SOS) relaxation~\cite{prajna2007framework}, which plays a central role in certifying polynomial nonnegativity. A polynomial 
$r(\xx)\in\mathbb{R}[\xx]$ is called an SOS polynomial if it admits a decomposition of the form
\begin{equation}\label{eq:sos_def}
    r(\xx) = \sum_{i=1}^{s} h_i^2(\xx),
\end{equation}
for some polynomials $h_i(\xx)\in\mathbb{R}[\xx]$.  
Equivalently,~\eqref{eq:sos_def} holds if and only if $r(\xx)$ can be represented via a Gram matrix:
\begin{equation}\label{eq:gram_form}
    r(\xx) = \mathbf{m}(\xx)^{\!\top} \, W \, \mathbf{m}(\xx),
\end{equation}
where $W$ is a symmetric positive semidefinite matrix ($W\succeq 0$), and $\mathbf{m}(\xx)$ is the vector collecting all monomials in $\xx$ up to degree 
$\frac{1}{2}\deg(r)$.
Thus searching for an SOS decomposition amounts to finding a feasible Gram matrix $W$ that satisfies~\eqref{eq:gram_form}. This leads to the semidefinite program:
\begin{equation*}\label{eq:sdp_relax}
\begin{aligned}
    &\text{find} && W  \\
    &\text{s.t.} && r(\xx) = \mathbf{m}(\xx)^{\!\top} W\, \mathbf{m}(\xx),\\
    &&& W \succeq 0,\quad W = W^{\!\top},
\end{aligned}
\end{equation*}
where one typically minimizes $\mathrm{trace}(W)$ as a dummy objective in the absence of a true optimization criterion.  
This SDP formulation can be efficiently solved by off-the-shelf SOS toolboxes such as {\it SOSTOOLS}~\cite{papachristodoulou2013sostools} and \textit{TSSOS}~\cite{wang2021tssos}.

\medskip
SOS relaxation is particularly useful for encoding nonnegativity of a polynomial over a basic semi-algebraic set.  
Consider the implication
\begin{equation*}\label{eq:implication}
    \bigwedge_{i=1}^{k} \, p_i(\xx) \ge 0  \Longrightarrow  q(\xx) \ge 0,
\end{equation*}
where $p_i(\xx), q(\xx) \in \mathbb{R}[\xx]$. The above machinery forms the basis for transforming the polynomial barrier certificate conditions into tractable SOS programs in our method.

% Putinar’s Positivstellensatz~\cite{put93} ensures that if the quadratic module generated by $\{p_1,\dots,p_k\}$ is archimedean and 
% $q(\xx) > 0$ on the set $\{\xx : p_i(\xx)\ge 0,\,1\le i\le k\}$,  
% then $q$ admits the certificate
% \begin{equation}\label{eq:sos_certificate}
%     q(\xx) = \sigma_0(\xx) + \sum_{i=1}^{k} \sigma_i(\xx)\, p_i(\xx),
% \end{equation}
% where each $\sigma_i(\xx)$ is an SOS polynomial.  
% Hence, the existence of an SOS representation of the form~\eqref{eq:sos_certificate} provides a sufficient condition for establishing~\eqref{eq:implication}.  
% In practice, one restricts the degrees of the multipliers $\sigma_i$ to lie within a prescribed bound, yielding a hierarchy of SDPs that approximate the positivity condition.

% The above machinery forms the basis for transforming the polynomial barrier certificate conditions into tractable SOS programs in our method.

\section{The framework of Generative BC synthesis}\label{framework}
In this section, we introduce a novel framework that leverages an LLM as the central component for the end-to-end synthesis of BCs.
Firstly, we design structured formal specifications that encode system dynamics and safety requirements, providing LLMs with an interpretable interface to understand the principles of formal safety verification. 
Secondly, to overcome the key challenge posed by the lack of high-quality training set from real dynamical systems, we design a reverse-engineering data generation procedure that starts from valid BCs and algorithmically derives compatible dynamical systems, thereby enabling large-scale synthesis of supervised samples and effective data augmentation.
Finally, to verify the correctness of LLM-generated candidate certificates, we employ SOS relaxation together with efficient LMIs feasibility checking. 
\begin{figure}[htbp]
  \centering
  \includegraphics[width=0.47\textwidth]{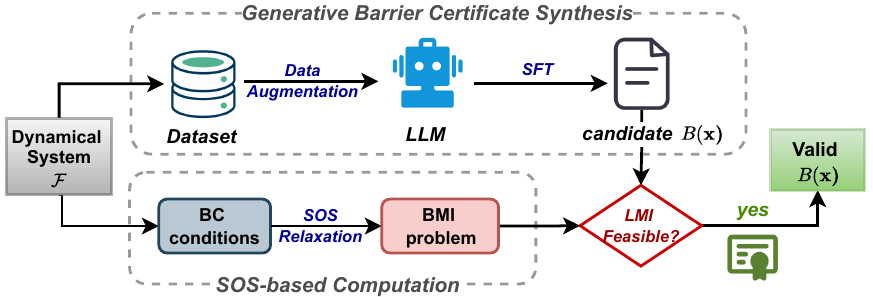}
  \caption{The pipeline of generative BC synthesis.}
  \label{fig:framework}
  % \Description{Logo of AAMAS 2026 -- The 25th International Conference on Autonomous Agents and Multiagent Systems.}
  \label{fig:framework_comparison}
\end{figure}
% This transforms the otherwise intractable verification task into an efficiently solvable, convex \textit{Linear Matrix Inequality (LMI)} feasibility problem. This verification pipeline demonstrates the potential of our framework as a powerful alternative to classical NP-hard Bilinear Matrix Inequality (BMI)-based methods.

As illustrated in Fig.~\ref{fig:framework_comparison}, our proposed generative BC synthesis method transforms the challenging NP-hard BMI-solving problem, which arises in traditional numerical approaches using SOS relaxation to compute BCs, into a more efficiently solvable convex LMI feasibility test problem. This not only enables the rapid generation of BCs with rich expressive power but also enhances the efficiency and scalability of the verification process. 

\subsection{Formulating Formal Specifications}
In this section, we first formulate structured formal specifications that serialize dynamical systems and safety requirements into a uniform symbolic format suitable for LLMs. This design transforms BC synthesis from a template-constrained search problem into an end-to-end sequence generation task, enabling LLMs to operate directly on symbolic representations.

Let us consider a dynamical system $\mathcal{F}:  \langle \ff, \Theta, \Psi \rangle$ whose vector fields $\ff(\xx)$ are polynomials, and the initial set is described as
$
\Theta = \{\xx \mid \theta_i(\xx) \ge 0\},
$ where $\theta$’s are vectors of polynomials and the inequalities are satisfied entry-wise. Thus, when $\Theta$ is a $n$-dimensional hypercube $[x_1^{\min}, x_1^{\max}] \times \dots \times [x_n^{\min}, x_n^{\max}]$, we may define
\[
{\theta}(\xx) = 
\begin{bmatrix} 
(x_1 - x_1^{\min})(x_1 - x_1^{\max}) \\
\vdots \\
(x_n - x_n^{\min})(x_n - x_n^{\max}) 
\end{bmatrix}.
\]
Similarly, we can also define the sets $\Psi$ and $\Xi$. Thus, each element of $\mathcal{F}$ is serialized into a structured textual representation, allowing LLMs to interpret and reason over the system symbolically. In a general form, this can be expressed as:
\begin{equation*}
\begin{aligned}
& \mathtt{-dimension: \, n}, \\
& \mathtt{-domain: \, D\_zones = [[\psi_1^{\min}, \psi_1^{\max}], \dots, [\psi_n^{\min}, \psi_n^{\max}]]}, \\
& \mathtt{-initial\_set: \, I\_zones = [[\theta_1^{\min}, \theta_1^{\max}], \dots, [\theta_n^{\min}, \theta_n^{\max}]]}, \\
& \mathtt{-unsafe\_set: \, U\_zones = [[\xi_1^{\min}, \xi_1^{\max}], \dots, [\xi_n^{\min}, \xi_n^{\max}]]}, \\
& \mathtt{-dynamics: \, dx\_k = f_k(x_1,\dots,x_n), \, k=1,\dots,n}, \\
& \mathtt{-barrier: \, B(x_1,\dots,x_n) = \text{symbolic polynomial expression}}.
\end{aligned}
\end{equation*}

These specifications forms a mathematically faithful interface between formal verification principles and the token-based reasoning capabilities of LLMs. As a result, the model operates over an expressive functional space
\[
\mathcal{F} := \{ B(\xx)\in \mathbb{R}[\xx] \;:\; \text{arbitrary
symbolic form} \},
\]
rather than being restricted to fixed-degree polynomial templates used in classical SOS-based numerical computations.

% \[
% \begin{aligned}
% B(\xx) &\le 0, && \forall \xx \in \Theta, \\
% B(\xx) &> 0, && \forall \xx \in \mathcal{X}_u, \\
% \nabla B(\xx)\cdot \ff(\xx) &\le 0, && \forall \xx.
% \end{aligned}
% \]

To illustrate the structured formal specification for a dynamical system, we present an example dataset entry. Each entry encodes the system dimension, the domain of each state variable, the initial and unsafe sets, the system dynamics, and a candidate BC in a structured, code-like format suitable for LLM input:
\lstset{
  backgroundcolor=\color{gray!10}, % 背景色
  basicstyle=\ttfamily\footnotesize, % 字体
  frame=single, % 框
  breaklines=true, % 自动换行
  columns=fullflexible
}

\begin{lstlisting}
Dataset Entry (Example1):
n = 2
D_zones = [[1.0, 5.0], [1.0, 5.0]]
I_zones = [[4.0, 4.5], [0.9, 1.1]]
U_zones = [[1.0, 2.0], [2.0, 3.0]]
f_expressions = [
    "x1**2 - 5.5*x1",
    "-x0**2 + 6*x0"
]
Barrier_expression = "-3.686*x0**2 - 3.156*x0*x1 + ..."
\end{lstlisting}

Within this representation, the LLM is trained to generate barrier certificates $B(\xx)$ satisfying the conditions in Theorem~\ref{th:BC}. This enables symbolic manipulation over a richer hypothesis space and supports the emergence of formal reasoning capabilities in automated safety verification.

% \subsection{Training Corpus Synthesis via Backward Construction and Diffeomorphic Augmentation}

\subsection{Training Corpus Construction and Augmentation}\label{data_set}

LLM supervised fine-tuning (SFT) requires a structurally diverse and semantically coherent training corpus, yet there are no publicly available resources tailored to this task. We therefore construct a high-quality seed corpus using an inverse-design strategy: instead of discovering an unknown $B(\xx)$ for a fixed system, we begin with a candidate certificate \( B(\xx) \), stochastically synthesized to satisfy certain prior conditions, and from it construct a system that is mathematically guaranteed to admit \( B(\xx) \) as a valid certificate, together with the corresponding state-space sets. 
This procedure allows systematic generation of self-consistent BC instances, which then serve as seed examples for the data augmentation.

\subsubsection{Backward Construction of Barrier Certificate Instances}
We randomly sample candidate certificates $B(\mathbf{x})$ from a controlled high-dimensional function family, specifically multivariate polynomials with random coefficients. By adjusting the degree $N_{\text{deg}}$ and coefficient range $\mathcal{I}_{\text{coef}}$, we control the complexity and geometric diversity of these functions.

For each sampled candidate $B(\xx)$, we construct system dynamics $f(\xx)$ that satisfy the Lie-derivative condition $\mathcal{L}_f B(\xx) \leq 0$. To strictly enforce this while enabling complex dynamics (e.g., rotational motion), we decompose $f(\xx)$ into a dissipative gradient flow and a conservative perturbation:
\[
    f(\xx) = \underbrace{-\nabla B(\xx)}_{\text{dissipative}} + \underbrace{S(\xx)\,\nabla B(\xx)}_{\text{conservative}},
\]
where $S(\xx)$ is a stochastically generated skew-symmetric matrix function. The first term drives the system along the negative gradient to decrease $B(\xx)$, while the second term—being orthogonal to $\nabla B(\xx)$—induces motion along the level sets without changing the value of $B(\xx)$.
Using the identity $v^\top S(\xx)v = 0$ for any vector $v$, the Lie-derivative becomes:
\begin{align*}
    \mathcal{L}_f B(\xx)
    &= \nabla B(\xx)^\top f(\xx) \nonumber \\
    &= -\|\nabla B(\xx)\|^2 + \nabla B(\xx)^\top S(\xx)\,\nabla B(\xx) \nonumber \\
    &= -\|\nabla B(\xx)\|^2 \leq 0.
    \label{eq:lie_derivative}
\end{align*}
This construction guarantees that $B(\xx)$ is a valid BC by design.

To complete the seed instance, the initial set $\Theta$ and unsafe set $\Xi$ are determined based on the topological properties of $B(\xx)$ within the domain $\mathcal{D}$. We use numerical optimization to locate the global minimum $\mathbf{x}^*$ and representative boundary maxima $\mathbf{x}_{\text{bd}}$. In accordance with the definition of BCs, these extremal points serve as anchors for constructing $\Theta$ and $\Xi$, around which bounded regions are defined to yield a well-posed verification problem. The process produces a complete training sample $\mathcal{F}: \langle f, \Theta, \Xi \rangle$ together with its solution \( B(\xx) \), forming the high-fidelity seed dataset.
%%%%%%%%%%%%%%%%%%%%%%%%%%%%%%%%%

\subsubsection{Data Augmentation via Constrained CDT}

Although the backward construction above yields self-consistent problem--solution pairs, the resulting dataset is still quite similar in structure: the construction template is highly regular, and the induced vector fields and sets tend to share a narrow range of geometric patterns. We therefore treat the backward-constructed instances as a compact but high-fidelity seed dataset, and expand them into a larger, more diverse dataset by applying coordinate transformations that preserve BC validity.

Starting from a valid instance
$
    \langle \mathcal{F}, B, \Theta, \Xi \rangle,
$
we apply a coordinate transformation \( \xx = \phi(\mathbf{y}) \) to obtain
$
    \langle \mathcal{F}', B', \Theta', \Xi' \rangle,
$
where \( B'(\mathbf{y}) = B(\phi(\mathbf{y})) \) and
\[
    \mathcal{F}'(\mathbf{y}) = J_{\phi}(\mathbf{y})^{-1}\, \mathcal{F}(\phi(\mathbf{y})), \quad
    J_{\phi}(\mathbf{y}) = \frac{\partial \phi}{\partial \mathbf{y}}.
\]
Requiring \( \phi \) to be a diffeomorphism ensures that
$
    \mathcal{L}_{\mathcal{F}'} B'(\mathbf{y})
    = \mathcal{L}_{\mathcal{F}} B(\xx),
$
so non-positivity of the Lie-derivative and the sign constraints on \( \Theta \) and \( \Xi \) are preserved exactly. 

To ensure that box-shaped regions remain boxes after the transformation, we introduce a constrained coordinate diffeomorphism transformation (CDT).
We restrict the affine map to
\[
    \mathbf{x} = \Lambda \mathbf{y} + \delta, \quad
    \Lambda = \mathrm{diag}(\lambda_{1}, \dots, \lambda_{n}), \; \lambda_{i} \neq 0,
\]
so that the transformation consists only of independent scaling along each coordinate, possible reflections (if some \( \lambda_{i} < 0 \)), and translation. Under this diagonal constraint, any box is mapped to a box, and the transformed bounds are computed in closed form by
\[
    y_i = \lambda_{i}^{-1}(x_i - \delta_i),
\]
applied to the endpoints of each interval.

The constrained CDT prioritizes data validity and a consistent representation over rotational or shear variations: each augmented instance remains semantically consistent with its original barrier certificate example and fits directly into the box format required by our SFT pipeline.

\begin{algorithm}[htbp]
\caption{Backward Construction and Constrained CDT Augmentation}
\label{alg:construction-augmentation}
\begin{algorithmic}[1]
\Require Degree $N_{\text{deg}}$, range $\mathcal{I}_{\text{coef}}$, domain $\mathcal{D}$, CDT bounds $\mathcal{I}_{\text{scale}}, \mathcal{I}_{\text{trans}}$
\Ensure Augmented instance $\mathcal{F}' = \langle f', \Theta', \Xi' \rangle$, certificate $B'$

\State $B(\mathbf{x}), S(\mathbf{x}) \gets \texttt{SamplePoly}(N_{\text{deg}}, \mathcal{I}_{\text{coef}})$
\State $f(\mathbf{x}) \gets -\nabla B(\mathbf{x}) + S(\mathbf{x})\nabla B(\mathbf{x})$
\State $\mathbf{x}^* \gets \arg\min_{\mathbf{x}\in\mathcal{D}} B(\mathbf{x}); \quad \mathbf{x}_{\text{bd}} \gets \texttt{BoundaryMax}(\mathcal{D}, B(\mathbf{x}))$
\State $\Theta \gets \texttt{Region}(\mathbf{x}^*); \quad \Xi \gets \texttt{Region}(\mathbf{x}_{\text{bd}})$
\State $\mathcal{F} \gets \langle f, \Theta, \Xi \rangle$ {\color{gray}{\Comment{Construct self-consistent seed instance}}}

\State $\Lambda \gets \texttt{diag}(\mathcal{U}(\mathcal{I}_{\text{scale}})^n); \quad \delta \gets \mathcal{U}(\mathcal{I}_{\text{trans}})^n$
\State $\phi(\mathbf{y}) \gets \Lambda \mathbf{y} + \delta; \quad J_{\text{inv}} \gets \Lambda^{-1}$
\State $B'(\mathbf{y}) \gets B(\phi(\mathbf{y}))$
\State $f'(\mathbf{y}) \gets J_{\text{inv}} f(\phi(\mathbf{y}))$ {\color{gray}{\Comment{Apply coordinate transformation to dynamics}}}

\For{$K \in \{\Theta, \Xi\}$}
    \State Let $[\mathbf{l}, \mathbf{u}]$ be the bounds of box set $K$
    \State $\mathbf{v}_1 \gets \Lambda^{-1}(\mathbf{l} - \delta); \quad \mathbf{v}_2 \gets \Lambda^{-1}(\mathbf{u} - \delta)$
    \State $K' \gets [\min(\mathbf{v}_1, \mathbf{v}_2), \max(\mathbf{v}_1, \mathbf{v}_2)]$
\EndFor

\State \Return $\langle f', \Theta', \Xi' \rangle, B'$
\end{algorithmic}
\end{algorithm}

The full workflow described above is formalized in Algorithm~\ref{alg:construction-augmentation}. Lines 1--2 initiate the backward construction by sampling $B(\mathbf{x})$ and $S(\mathbf{x})$ to synthesize the vector field $f(\mathbf{x})$ satisfying the stability constraint. Lines 3--4 locate the extrema $\mathbf{x}^*$ and $\mathbf{x}_{\text{bd}}$ to instantiate the consistent seed sets $\Theta$ and $\Xi$. Subsequently, Lines 5--7 parameterize the constrained CDT via $\Lambda$ and $\delta$ to derive the augmented pair $\langle f', B' \rangle$. Finally, Lines 8--11 perform the coordinate update of box boundaries $\mathbf{l}, \mathbf{u}$, employing element-wise operations to yield the valid augmented sets $\Theta', \Xi'$.

\subsection{Identify a Real Barrier Certificate}
This subsection examines the SOS relaxation method~\cite{cav2021} to identify a real BC. 
Given a continuous system $\mathcal{F}: \langle \ff, \Theta, \Psi \rangle$ with unsafe set $\Xi$, the LLM-generated candidate $B(\xx)$ may not satisfy the conditions of Theorem~\ref{th:BC} over the full continuous domain. We therefore focus on determining whether this candidate is indeed valid. This verification reduces to solving a LMI feasibility test problem~\cite{scherer2000linear}.

%As shown in the \emph{verifier} of Fig.~\ref{framework}, 
%we first transfer the problem of checking the validation of the conditions of barrier certifications to an LMI feasibility solving problem based on  SOS relaxation. 

%relaxation technology to encoding a set of constraints that ensure the strict correctness of BC according to Definition 2 and Positivestelensatz theorem. This allows us to transform the verification process into an optimization problem. Since the candidate BC is already known, the optimization problem can be simplified to an LMI feasibility decision, which we can  be solved using the \emph{MOSEK} solver. 

%%%%%%%%%%%%%%%%
% Following Theorem~\ref{th:BC}, to identify $B(\xx)$ is a real barrier certificate, we need to verify the validation for three verification conditions, that is, $B(\xx)$ must satisfy the following conditions
% \begin{eqnarray}\label{problem:barrierexist}
% \left.\begin{array}{l@{}l}
% &B(\xx)\leq 0, \,\,\,\forall \xx\in \Theta, \\
% &\mathcal{L}_f B(\xx)-\lambda(\xx)B(\xx)< 0, \quad\forall \xx\in \Psi,\\
% &B(\xx) > 0, \,\,\,\forall \xx\in \Xi.\\
% \end{array}\right\}
% \end{eqnarray}
% Investigating (\ref{problem:barrierexist}), it turns out that all constraints can be encoded as nonnegativity constraints for polynomials over the associated semi-algebraic sets.
%%%%%%%%%%%%%%%%

% In what follows, we recap how to  transform  
% the verification problem for the polynomial's nonnegativity over 
% a semi-algebraic set,  into an LMI feasibility test problem. 

Let $S \subset \RR^{n}$ be a nonempty basic semi-algebraic set, defined by
\begin{equation}\label{eq:res_k}
S:=\{\xx \in \RR^{n}\,|\,p_1(\xx)\geq 0,\ldots, p_{k}(\xx)\geq 0\},
\end{equation} 
where $p_{i}(\xx) \in \RR[\xx], 1\leq i \leq k$. Let $\pp=(p_1,p_2,\ldots,p_k)$
be a tuple of polynomials in $\RR[\xx]$. The polynomial set defined by
$$M(\pp)=\{\sigma_{0}+\sum_{j=1}^{k}\sigma_{j} \, p_{j} |\, \sigma_{j}\in\Sigma[\xx], 0\leq j \leq k\} $$
is called the quadratic module generated by $\pp$.
The positivity of polynomials over semi-algebraic sets can be investigated via the following Putinar's Positivstellensatz theorem~\cite{put93}.
\begin{theorem}[Putinar’s Positivstellensatz~\cite{put93}]\label{th:positive} Let $S \subset \RR[\xx]$ be as in
(\ref{eq:res_k}). Assume that the quadratic module $M(\pp)$ is
archimedean, namely, there exists $u(\xx) \in M(\pp)$ such that the
 level set $\{\xx \in \RR^{n} | u(\xx) \geq 0\}$ is compact.  If $f(\xx)\in \RR[\xx]$
is strictly positive on $S$, then $f(\xx) \in M(\pp)$, that is, 
$f(\xx)$ can be represented as
\begin{equation}\label{eq:f_positive}
      f(\xx)=\sigma_{0}(\xx)+\sum_{j=1}^{k}\sigma_{j}(\xx)p_{j}(\xx),
\end{equation}
where $\sigma_{j} \in \Sigma[\xx], 0\leq j\leq k$.
\end{theorem}

From Theorem~\ref{th:positive}, the existence of the representation
as in (\ref{eq:f_positive}) provides a sufficient and necessary condition for the strict positivity of $f(\xx)$ on $S$. It is known 
that the degree bound for $\sigma_{j}(\xx)$ is exponential with 
$n$ and $\deg(f)$. To illustrate the computational applicability, we turn to find the above representation by fixing a priori bound $2d$ on the degree of the multipliers $\sigma_{j}(\xx)$, namely, $\sigma_{j}\in \Sigma_{d}[\xx]$. 
%According to  Putinar's Positivstellensatz, 
%For instance, let us consider the first constraint in (\ref{problem:barrierexist}).  
%Stated as above, $\Theta$ is a basic semi-algebraic set defined by 
%$$\Theta:&=\{\xx \in \RR^{n}\,|\,\theta_1(\xx)\geq 0,\ldots, \theta_{q}(\xx)\geq 0\}.$$  
Thus, a sufficient condition for $f(\xx)$ 
is nonnegative on the semi-algebraic set $S$ is that 
there exist SOS polynomials $\sigma_{j}(\xx) \in \RR[\xx]$ for $j=0,\ldots, k$, such that $f(\xx)$ can be written as 
\begin{equation}\label{eq:bound_D}
f(\xx)=\sigma_{0}(\xx)+\sum_{j=1}^{k} \sigma_{j}(\xx) p_{j}(\xx),
\,\,\, \text{with}\,\,\, \sigma_{j} \in \Sigma_{d}[\xx], \, 0\leq j \leq k. 
\end{equation}
%with $\deg(\sigma_{j})\leq D$, 
Moreover, given a polynomial $f(\xx)$ is nonnegative on $S$, checking whether $f(\xx)$ has the representation (\ref{eq:bound_D}) reduces to solving a semidefinite program problem.  
%The problem of finding SOS polynomials $\sigma_{j}(\xx)$
%for $j=0,\ldots,k$ can be further converted into the following semi-definite programming (SDP) problem
%\begin{equation}\label{SOSW}
%\left.\begin{array}{l@{}l}
%\displaystyle \inf_{W_{j}} \ & \text{Trace}(W_{j}) \\
%\text{s.\ t.} \ &  \sigma_{j}(\xx)=\mm_{j}(\xx)^T \cdot W_{j} \cdot  \mm_{j}(\xx) \smallskip\\
%& W_{j} \succeq 0, W_{j}^T = W_{j},
%\end{array}\right\}
%\end{equation}
%where $\text{Trace}(W_{j})$ acts as a dummy objective function that is
%commonly used in SDP for optimization problem with no objective
%functions, $\mm_j(\xx)$ is a vector of all monomials in~$\RR[\xx]$ with degree $\leq \frac{1}{2}
%\deg \sigma_{j}(\xx)$.

% Investigating~(\ref{problem:barrierexist}), 
% all conditions  %in~(\ref{problem:barrierexist})
% can be derived as a unified type, that is, 
% polynomial nonnegativity on a semi-algebraic set. 
% From Theorem~\ref{th:positive}, the verification conditions as in~(\ref{problem:barrierexist}) can be 
% relaxed as more tractable ones by the representation as (\ref{eq:bound_D}).
Investigating Theorem~\ref{th:BC}, 
all conditions  %in~(\ref{problem:barrierexist})
can be derived as a unified type, that is, 
polynomial nonnegativity on a semi-algebraic set. 
From Theorem~\ref{th:positive}, the verification conditions as in~Theorem~\ref{th:BC} can be 
relaxed as more tractable ones by the representation as (\ref{eq:bound_D}).

\begin{theorem}\label{th:BC_relax} %[Barrier Certificate]\label{prajna2007safety}
Let $\mathcal{F}: \langle \ff, \Theta, \Psi \rangle $ be a continuous system, and an unsafe state set $\Xi$ be defined as above. 
Let $B(\xx)$ be a given polynomial and $d$ be a positive integer. 
If there exist 
%a polynomial $\lambda(\xx)$, 
sums of squares $\sigma_i(\xx)\in \Sigma_{d}[\xx]$, 
$\phi_{i}(\xx)\in \Sigma_{d}[\xx]$, 
$\delta_{i}(\xx)\in \Sigma_{d}[\xx]$, 
and a polynomial $\lambda(\xx)$ with $\deg(\lambda)\leq 2d$, 
and positive numbers $\epsilon_1, \epsilon_2$, such that
the following conditions hold:
  % a real-valued function $B: \Psi\rightarrow \RR$,
%which satisfies the following conditions: 
%a barrier certificate of $\mathcal{F}$ for safety with respect to the unsafe region $\Xi$ is a real-valued function $B: \Psi\rightarrow\RR$,
\begin{enumerate}
%\item[(i)]
\item 
 $-B(\xx)-\sum_{i=1}^{q}\sigma_{i}(\xx)\theta_{i}(\xx)\in\Sigma_{d}[\xx]$,
%\item[(ii)] 
\item
$-\mathcal{L}_f B(\xx)+\lambda (\xx)B(\xx)-\sum_{i=1}^{r}\phi_i(\xx)\psi_{i}(\xx)-\epsilon_1\in \Sigma_{d}[\xx]$,
% \item[(iii)] 
\item
$B(\xx)-\sum_{i=1}^{p}\delta_{i}(\xx)\xi_{i}(\xx)-\epsilon_2\in\Sigma_{d}[\xx]$,
\end{enumerate}
%where $\mathcal{L}_f B(\xx)$ denotes the {Lie-derivative} of %$B(\xx)$ along the vector field $\ff(\xx)$, i.e., $\mathcal{L}_f %B(\xx)=\sum_{i=1}^n\frac{\partial B}{\partial x_i} \cdot f_{i}%%(\xx)$. In this paper, the third condition is relaxed to %$\mathcal{L}_f B(\xx)-\lambda B(\xx)<0\quad\forall \xx\in \Psi$, %which is convex.
Then $B(\xx)$ is a BC of the system $\mathcal{F}$, and the safety of $\mathcal{F}$ is guaranteed.
\end{theorem}
%transformed into an equivalent linear matrix inequality feasibility of the form
\begin{proof}
As shown in (\ref{eq:f_positive}), sum-of-squares representation 
indicates that the conditions (1-3) can imply the conditions (i-iii) in Theorem~\ref{th:BC}, respectively. Thus, 
the claim is proved.
\end{proof}
The conditions (1-3) in Theorem~\ref{th:BC_relax}
can be rewritten as an LMI feasibility test problem, namely,    
%the learned neural barrier function is represented in a polynomial form, we use the SOS technique to transform its safety verification problem~(\ref{problem:barrierexist}) into the following optimization problem:
\begin{eqnarray}\label{problem:LMI}
\left.\begin{array}{l@{}l}
&\rm {find} \quad \sigma_{i}(\xx), \delta_{i}(\xx), \phi_{i}(\xx), \lambda(\xx)\\ 
&\text{s.t.}\,\,\,\, -B(\xx)-\sum_{i=1}^{q}\sigma_{i}(\xx)\theta_{i}(\xx)\in\Sigma_{d}[\xx],\\
&\quad\quad -\mathcal{L}_f B(\xx)+\lambda (\xx)B(\xx)-\sum_{i=1}^{r}\phi_i(\xx)\psi_{i}(\xx)-\epsilon_1\in\Sigma_{d}[\xx],\\
&\quad \quad 
B(\xx)-\sum_{i=1}^{p}\delta_{i}(\xx)\xi_{i}(\xx)-\epsilon_2\in\Sigma_{d}[\xx],\\
\end{array}\right\}
\end{eqnarray}
where $\epsilon_1, \epsilon_2 \in \RR_{>0}$ are prespecified small positive numbers, and $\sigma_{i}(\xx) \in \Sigma_{d}[\xx]$, 
$\delta_{i}(\xx)\in \Sigma_{d}[\xx]$, $\phi_{i}(\xx)\in \Sigma_{d}[\xx]$, 
and $\lambda(\xx)$ is a polynomial with $\deg(\lambda)\leq 2d$. 
Many SDP solvers such as {\it MOSEK}~\cite{andersen2000mosek}, {\it SOSTOOLS}~\cite{prajna2004sostools}, and {\it YALMIP}~\cite{lofberg2004yalmip} enable efficient solution of the above SDP problems.

\begin{table*}[ht]
\caption{Performance Evaluation on Benchmark Examples}
\label{tab:comparison_with_penbmi_fullwidth}
\centering
\renewcommand{\arraystretch}{0.97}
\begin{tabular}{rcccccc}
\toprule
\multirow{2}{*}{\textbf{Benchmark}} 
& \multicolumn{2}{c}{\textbf{LLM (Ours)}} 
& \multicolumn{2}{c}{\textbf{PENBMI~\cite{kocvara2005penbmi}}} 
& \multicolumn{2}{c}{\textbf{FOSSIL~2.0~\cite{edwards2024FOSSIL}}} \\
\cmidrule(lr){2-3} \cmidrule(lr){4-5} \cmidrule(lr){6-7}
& \textit{Success Rate (\%)} & \textit{Avg. Runtime (s)} 
& \textit{Success Rate (\%)} & \textit{Avg. Runtime (s)} 
& \textit{Success Rate (\%)} & \textit{Avg. Runtime (s)} \\ 
\midrule
2D Systems & 81.6 & 0.5499 & 54.8 & 1.2714  & 95.2 & 1.2660 \\
3D Systems & 61.6 & 1.0537 & 38.80 & 3.2293 & 73.9 & 4.0685 \\
4D Systems & 50.8 & 1.7072 & 24.80 & 38.3319 & 41.7 & 16.8028 \\
5D Systems & 50.0 & 3.9809 & 27.20 & 55.2020 & 24.3 & 35.0049 \\ 
\midrule
\textbf{Overall}\;\;\; & 61.0 & 1.8229 & 36.4 & 24.5086 & 58.7 & 14.2856 \\ 
\bottomrule
\end{tabular}
\end{table*}

% \begin{table*}[ht]
% 	\caption{Performance Evaluation on Benchmark Examples}
% 	\label{tab:comparison_with_penbmi_fullwidth}
% 	\centering
% 	% 增加行高，1.2表示为默认行高的1.2倍
% 	\renewcommand{\arraystretch}{1.2}
% 	\begin{tabular*}{\textwidth}{@{\extracolsep{\fill}} r cccc} \toprule
% 		& \multicolumn{2}{c}{Ours} & \multicolumn{2}{c}{PENBMI} \\
% 		\cmidrule(lr){2-3} \cmidrule(lr){4-5}
% 		\textit{System Dimension} & \textit{Success Rate (\%)} & \textit{Avg. Runtime (s)} & \textit{Success Rate (\%)} & \textit{Avg. Runtime (s)} \\ \midrule
% 		2D Systems & 81.6 & 0.2499 & 54.8 & 1.2714 \\
% 		3D Systems & 61.6 & 0.6537 & 38.80 & 3.2293 \\
% 		4D Systems & 50.8 & 1.2072 & 24.80 & 38.3319 \\
% 		5D Systems & 50.0 & 3.5809 & 27.20 & 55.2020 \\ \midrule
% 		\textbf{All} & 61.0 & 1.4229 & 36.4 & 24.5086 \\ \bottomrule
% 	\end{tabular*}
% \end{table*}

\section{Experiments}\label{exp}

In this section, we conduct a set of experiments to systematically evaluate the effectiveness and practical viability of our proposed LLM-based generative BC synthesis method. We begin with a case study on a real-world nonlinear system, offering a step-by-step demonstration of the complete workflow. We then provide a quantitative comparison against the BMI solver {\it PENBMI}~\cite{kocvara2005penbmi} and the state-of-the-art neural BC synthesis tool {\it FOSSIL~2.0}~\cite{edwards2024FOSSIL}, assessing both success rate and execution time across systems of different dimensions. All experiments were carried out on an Intel(R) Xeon(R) Gold 6246 CPU @ 3.30GHz equipped with two NVIDIA GeForce RTX 2080 Ti GPUs. Source code and experimental details are available at~\href{https://anonymous.4open.science/r/llm4bc-9782}{https://anonymous.4open.science/r/llm4bc-9782}.
\subsection{Experiment Setup}

{\bf{[Dataset Construction].}} The training dataset is generated using the backward construction technique detailed in the Section~\ref{data_set}, comprising a total of 100000 problem-solution pairs. These synthesized dynamical systems cover a range of dimensions, from 2D to 5D, ensuring both diversity and complexity within the corpus. For evaluation, a separate test set of 1000 cases is generated using the same procedure. Critically, to guarantee a fair evaluation and prevent any data leakage, the generation processes for the training and test sets were seeded with different random states, ensuring the test set is a true hold-out collection.

{\bf{[Model and Fine-tuning].}} We selected {\it Qwen3-0.6B} as the base architecture and performed full-parameter supervised fine-tuning on our constructed dataset of 100000 samples. The key hyperparameters for the fine-tuning process were configured as follows: the learning rate was set to \(1.0 \times 10^{-5}\), and a warmup ratio of $0.1$ is employed. During the inference phase for generating certificates, the decoding temperature is set to $0.2$ and $top$-$p$ is set to 1.0 to produce low-variance outputs.

% {\bf{[Verification Setup and Baseline].}} To ensure a fair performance comparison, we standardized key parameters in the verification process. In the post-verification stage for our method, the degree for all SOS multipliers utilized in the three LMI feasibility problems was uniformly set to $2$, and solved by {\it MOSEK}~\cite{andersen2000mosek}.
% As our baseline, we selected {\it PENBMI}~\cite{kocvara2005penbmi}, a prominent and widely-used solver for BMI problems. The baseline experiments were conducted in a MATLAB R2022b environment, where the barrier Certificate synthesis problems were modeled using the {\it YALMIP}~\cite{lofberg2004yalmip} toolbox and subsequently solved by {\it PENBMI}. For each case in our test set, {\it PENBMI} was tasked with solving its traditional BMI formulation, which involves a simultaneous search for the coefficients of a polynomial certificate $B(\xx)$ and its corresponding SOS multipliers. Correspondingly, the degree of all SOS multipliers in this BMI formulation was also fixed to $2$. A timeout of $3600$ seconds was imposed for each test case. A run was deemed successful if {\it PENBMI} converged to a feasible solution within this time limit.

{\bf{[Verification Setup and Baseline].}} To ensure a fair performance comparison, we standardize key parameters in the verification process. In the post-verification stage for our method, the degree for all SOS multipliers utilized in the three LMI feasibility problems is uniformly set to $2$, and solved by {\it MOSEK}~\cite{andersen2000mosek}.
As our baseline, we select {\it PENBMI}~\cite{kocvara2005penbmi}, a prominent and widely-used solver for BMI problems. For each case in our test set, {\it PENBMI} is tasked with solving its traditional BMI formulation, which involves a simultaneous search for the coefficients of a polynomial certificate $B(\xx)$ and its corresponding SOS multipliers. Correspondingly, the degree of all SOS multipliers in this BMI formulation is also fixed to $2$. A timeout of $3600$ seconds is imposed for each test case. A run is deemed successful if {\it PENBMI} converged to a feasible solution within this time limit.

% Additionally, we also consider neural BC synthesis method {\it FOSSIL~2.0} as a baseline to emphasize the performance differences between LLMs and conventional NNs in BC learning.
Additionally, we consider the neural BC synthesis method {\it FOSSIL~2.0} as a baseline to highlight the performance gap between LLMs and conventional NNs in BC learning. For each test system, we evaluate 10 pre-specified parameter sets and consider the case successful if any yields a valid barrier certificate.

\subsection{Case Study}
% To concretely illustrate the end-to-end workflow of our proposed method, this section details the entire process for a representative three-dimensional non-linear system, from certificate generation by the LLM to its successful formal verification.
% 为了具体说明我们提出的方法的端到端工作流程，本节详细介绍了代表性三维非线性系统的整个过程，从LLM生成证书到成功的形式验证。
%重写refine并给出英文翻译，符合学术论文写作风格”为了进一步说明我们端到端大模型综合生成式BC的实用性，本节用一个来自文献中的真实示例来测试整个过程，从LLM生成证书到成功真实性形式化验证“

% To further demonstrate the utility of our end-to-end LLM-integrated generative BC framework, 
% this section details the entire process for a representative three-dimensional non-linear system, from certificate generation by the LLM to its successful formal verification
To concretely illustrate the end-to-end workflow of our proposed method, this section details the entire process for a real-world 3D example in the literature from certificate generation by the LLM to its successful formal verification.
% \paragraph{\textbf{Step 1: Problem Formulation}}

% \paragraph{\textbf{Step 1: Problem Formulation}}

% We consider a 3D non-linear dynamical system with state variables \( \xx = [x_1, x_2, x_3]^T \). The system (named 'C8') is defined by the following set of ordinary differential equations, which includes a cubic non-linearity:

\emph{Example 1.}~\cite{C9} Consider the following continuous nonlinear  dynamical system in the plant:
\begin{equation}
   \begin{bmatrix}
        {\dot{x_1}}\\
        {\dot{x_2}}\\
        {\dot{x_3}}\\
    \end{bmatrix}
    =
    \begin{bmatrix}
        -x_2\\
        -x_3 \\
        -x_1-2x_2-x_3+x_1^3\\
    \end{bmatrix}.
\end{equation}

The system domain is defined as $\Psi$=$\{\xx$=$(x_1,x_2,x_3)^T\in\mathbb{R}^3|-2\leq x_1,x_2,x_3\leq2\}$. Our goal is to confirm that all trajectories of the dynamical system, starting from the initial set
$
\Theta$=$\{\xx$=$(x_1,x_2,x_3)^T\in \mathbb{R}^3|-0.25\leq x_1\leq 0.75,-0.25\leq x_2\leq 0.75, -0.75\leq x_3 \leq 0.25\}
$, will never enter the unsafe region
$
\Xi$=$\{\xx$=$(x_1,x_2,x_3)^T\in \mathbb{R}^3|1\leq x_1\leq 2,-2\leq x_2\leq -1, -2\leq x_3 \leq -1\}.
$

% We consider a 3D non-linear dynamical system~\cite{C9}:
% \begin{equation*}
% \begin{cases}
%      \begin{aligned}
%          \dot{x}_1 &= -x_2 \\
%          \dot{x}_2 &= -x_3 \\
%          \dot{x}_3 &= -x_1 - 2x_2 - x_3 + x_1^3.
%      \end{aligned}
% \end{cases}
% \end{equation*}
% The safety verification task is defined over the domain \( D = [-2, 2]^3 \). The Initial Set \( I \) is given by the hyperrectangle:
% \[ I = [-0.25, 0.75] \times [-0.25, 0.75] \times [-0.75, 0.25] \]
% The Unsafe Set \( U \) is given by:
% \[ U = [1, 2] \times [-2, -1] \times [-2, -1] \]

% \emph{Example 1.[Academic 3D Model]} Consider the following controlled continuous dynamical system in the plant:
% \begin{equation}
%    \begin{bmatrix}
%         {\dot{x}}\\
%         {\dot{y}}\\
%         {\dot{z}}\\
%     \end{bmatrix}
%     =
%     \begin{bmatrix}
%         z+8y\\
%         -y+z \\
%         -z-x^2+u\\
%     \end{bmatrix}.
% \end{equation}

% The system domain is defined as $\Psi$=$\{\xx$=$(x,y,z)^T\in\mathbb{R}^3|-2.2\leq x,y,z\leq2.2\}$. Our goal is to confirm that all trajectories of the closed-loop system under the control $u$=$k(x,y,z)$, starting from the initial set
% $
% \Theta$=$\{\xx$=$(x,y,z)^T\in \mathbb{R}^3|-0.4\leq x,y,z\leq 0.4\}
% $, will never enter the unsafe region
% $
% \Xi$=$\{\xx$=$(x,y,z)^T\in \mathbb{R}^3|2\leq x,y,z\leq 2.2\}.
% $

% \paragraph{\textbf{Step 2: Certificate Generation via LLM}}
The structured text description of this system is provided as input to our fine-tuned {\it Qwen3-0.6B} model~\cite{bai2023qwen}. In a single forward pass, the model directly generates the following quadratic polynomial as a candidate barrier certificate, denoted as :
\begin{align*}
B(\xx) = & 0.678x_1^2 - 0.504x_1x_2 - 1.116x_1x_3 - 0.939x_1 - 0.209x_2^2\\
+& 2.008x_2x_3- 2.614x_2 + 0.075x_3^2 + 0.311x_3 - 3.320
\end{align*}

% \paragraph{\textbf{Step 3: Post-Verification via SOS}}
The generated symbolic candidate $B(\xx)$ is subsequently passed to our post-verification module. This module transforms the three requisite conditions for a valid BC into three independent SOS feasibility problems. By invoking a convex optimization solver {\it MOSEK}, we confirm that feasible solutions exist for all three SOS problems, which provides a rigorous proof that the function $B(\xx)$ generated by the LLM is a valid BC.

The entire automated process, from inputting the system description to obtaining a formally verified solution, demonstrates the practical capability of our framework in tackling challenging non-linear verification tasks.

\subsection{Performance Evaluation}

% This section provides a quantitative analysis of the performance of our proposed framework for BC generation using a LLM. 

% This section provides a quantitative analysis in terms of effectiveness (\textit{success rate}) and efficiency (\textit{runtime}) of the performance of our proposed LLM-based framework for BC generation, comparing it against the BMI-based numerical solver {\it PENBMI} and the SMT-based neural BC synthesis tool {\it FOSSIL~2.0}. The experimental results are presented in Table~\ref{tab:comparison_with_penbmi_fullwidth}. 

This section presents a quantitative evaluation of our LLM-based framework for BC generation, assessing both effectiveness (\textit{success rate}) and efficiency (\textit{average runtime}). We compare our method against the BMI-based numerical solver {\it PENBMI} and the SMT-guided neural BC synthesis tool {\it FOSSIL~2.0}. The experimental results are presented in Table~\ref{tab:comparison_with_penbmi_fullwidth}.

% A direct comparison with the traditional PENBMI solver indicates that our approach shows notable advantages in terms of both effectiveness (\textit{success rate}) and efficiency (\textit{runtime}).

% Regarding effectiveness, our method demonstrates a higher success rate in finding valid solutions compared to the baseline. Overall, our model achieved a total success rate of 61.0\% across the 1000 test cases, whereas {\it PENBMI}'s success rate was 36.4\%. This performance gap is generally maintained across systems of varying dimensions. For instance, on 2D systems, our method's success rate was 81.6\%, compared to {\it PENBMI}'s 54.8\%. As the system dimension increases, the success rates of both methods naturally decline, yet the performance degradation of our approach is less pronounced. On the 5D systems, our method maintains a 50.0\% success rate, while {\it PENBMI}'s performance drops to 27.2\%. This suggests that our generative paradigm may possess greater robustness when tasked with exploring and discovering valid solutions for complex, high-dimensional problems.

Regarding effectiveness, our model achieved a total success rate of 61.0\% across the 1000 test cases, whereas {\it PENBMI}'s success rate was 36.4\%. This performance gap is generally maintained across systems of varying dimensions. For instance, on 2D systems, our method's success rate was 81.6\%, compared to {\it PENBMI}'s 54.8\%. 
% Notably, the SMT-based neural BC synthesis tool {\it FOSSIL~2.0} outperforms our approach on low-dimensional examples, benefiting from the stronger capabilities of SMT-guided post-processing in simpler settings. 
Notably, the SMT-based neural BC synthesis tool {\it FOSSIL~2.0} exhibits excellent performance on low-dimensional systems, achieving a 95.2\% success rate on 2D examples. This superiority is consistent with its SMT-guided post-verification mechanism, which is particularly effective when the search space is small and symbolic reasoning remains tractable. However, as system dimensionality increases, the success rate of {\it FOSSIL~2.0} decreases sharply—from 73.9\% in 3D to 41.7\% in 4D, and further down to 24.3\% in 5D. This degradation reflects the inherent difficulty of SMT-based refinement in high-dimensional continuous domains. 
% In contrast, our generative approach shows a more gradual decline, maintaining a 50.0\% success rate in 5D.
% However, as the system dimension increases, the success rate of the neural BC method declines sharply below that of our approach, while its runtime exceeds ours by more than an order of magnitude. 
In contrast, on 5D systems, our method maintains a 50.0\% success rate, whereas {\it PENBMI}'s drops to 27.2\%, and {\it FOSSIL~2.0} similarly suffers a significant performance loss. These results indicate that our generative paradigm exhibits greater robustness for exploring and discovering valid solutions in complex, high-dimensional problems.

% In terms of efficiency, the two methods exhibit a clear difference in scalability. The overall average runtime for our method is 1.42 seconds per case, compared to 24.51 seconds for {\it PENBMI}. More critically, the methods show different trends in computational cost as the system dimension grows. The runtime of our approach shows a moderate increase, from 0.25 seconds for 2D systems to 3.58 seconds for 5D systems. In contrast, {\it PENBMI}'s runtime increases more rapidly, which is particularly evident in the transition from 3D to 4D systems, where its average runtime increases from 3.23 seconds to 38.33 seconds. For 4D systems, the runtime of our method (1.21 seconds) is substantially lower than that of {\it PENBMI}.

In terms of efficiency, the two methods exhibit a clear difference in scalability. The overall average runtime for our method is 1.82 seconds per case, compared to 24.51 seconds for {\it PENBMI}. More critically, the methods show different trends in computational cost as the system dimension grows. The runtime of our approach shows a moderate increase, from 0.54 seconds for 2D systems to 3.98 seconds for 5D systems. In contrast, {\it PENBMI}'s runtime increases more rapidly, particularly in the transition from 3D to 4D systems, where its average runtime jumps from 3.23 seconds to 38.33 seconds. For 4D systems, the runtime of our method (1.70 seconds) is substantially lower than that of {\it PENBMI}.
% Moreover, it can be observed that both neural network-based and LLM-based BC synthesis methods achieve significantly higher efficiency than traditional numerical solvers. 
The SMT-based neural BC tool {\it FOSSIL~2.0} is also more efficient than {\it PENBMI}, achieving an average runtime of 1.27 seconds in 2D and 4.07 seconds in 3D. 
% However, its computational cost rises steeply as dimensionality increases: runtime grows to high dimensions exceeding our method by more than an order of magnitude. 
However, its computational cost escalates sharply with dimensionality, eventually exceeding ours by more than an order of magnitude in 4D and 5D ones.
% This pattern reflects the increasing burden of SMT-based post-processing in high-dimensional continuous domains.
% Overall, both neural network–based and LLM-based BC synthesis methods offer substantial efficiency advantages over traditional BMI solvers.
% This pattern reflects the scalability of the neural BC approach remains limited, while our LLM-driven generative framework maintains consistently low computational cost across dimensions, demonstrating superior efficiency and scalability for higher-dimensional systems
It reflects the scalability of the neural BC approach remains limited, while our LLM-driven generative framework is more efficient and scalable for higher-dimensional systems.

% However, the neural BC approach exhibits lower scalability compared to our LLM-based method, with runtime costs increasing sharply as system dimension rises.

% This observed difference in efficiency and scalability can be attributed to the fundamental theoretical divergence between the two methodologies. As a traditional solver, {\it PENBMI} is tasked with directly solving an BMI problem, for which computational tractability is not theoretically guaranteed. Our Generate-and-Verify framework, however, leverages the LLM to front-load the most difficult step: determining the certificate's symbolic structure. This architectural choice transforms the subsequent verification task from a non-convex BMI problem into a convex LMI feasibility problem, which is solvable in polynomial time. Therefore, the efficiency gains observed in our experiments are likely a manifestation of the theoretical advantages afforded by circumventing the computational bottleneck of conventional BMI solvers. In summary, the results suggest that our method is a promising alternative to traditional approaches, with notable potential in performance and scalability.
\section{Conclusion}

This paper introduced a generative framework for BC synthesis that replaces traditional numerical optimization with symbolic generation by LLM. By producing high-quality BC candidates directly, the proposed method transforms nonconvex BMI formulations into efficient LMI feasibility checks, greatly reducing computational cost while preserving correctness. Unlike system-specific neural BC approaches, our fine-tuned LLM demonstrates strong cross-system generalization, generating diverse and valid certificates for heterogeneous nonlinear systems without retraining. Experiments show substantial gains in both efficiency and success rate compared with classical SOS-based solver and learning-based neural method. 
% Overall, this work highlights the potential of generative AI to advance formal safety verification, offering a scalable and effective pathway for constructing safety guarantees in complex dynamical systems.

In future work, we plan to develop dedicated benchmarks tailored for training LLMs in formal verification tasks, and to explore advanced data augmentation and prompt engineering strategies to further enhance the capabilities of LLMs in verifying more complex dynamical systems.

\bibliographystyle{ACM-Reference-Format}
\bibliography{main}

\end{document}